\documentclass[11pt]{amsart}
\usepackage{epsfig,url}
\usepackage{mathdots}
\usepackage{float}

\newtheorem{theorem}{Theorem}[section]
\newtheorem{lemma}[theorem]{Lemma}

\newtheorem{corollary}[theorem]{Corollary}
\newtheorem{proposition}[theorem]{Proposition}

\theoremstyle{definition}
\newtheorem{definition}[theorem]{Definition}
\newtheorem{example}[theorem]{Example}

\newtheorem{note}[theorem]{Note}

\theoremstyle{remark}

\begin{document}

\title[Identities for generalized {E}uler polynomials]
{Identities for generalized {E}uler polynomials}

\author{Lin Jiu}
\address{Department of Mathematics,
Tulane University, New Orleans, LA 70118}
\email{ljiu@tulane.edu}

\author{Victor H. Moll}
\address{Department of Mathematics,
Tulane University, New Orleans, LA 70118}
\email{vhm@tulane.edu}

\author{Christophe Vignat}
\address{LSS-Supelec, Universit\'{e} Orsay Paris Sud 11, France}
\email{cvignat@tulane.edu}
\email{christophe.vignat@u-psud.fr}

%    General info

\date{\today}


\begin{abstract}
For $N \in \mathbb{N}$, let $T_{N}$ be the Chebyshev polynomial of the first kind. Expressions for the sequence 
of numbers $p_{\ell}^{(N)}$, defined as the coefficients in the 
expansion of $1/T_{N}(1/z)$,  are provided. These coefficients give formulas for the classical Euler polynomials 
in terms of the so-called generalized Euler polynomials. The proofs are based on a probabilistic interpretation of the 
generalized Euler polynomials recently given by Klebanov et al.  Asymptotics of $p_{\ell}^{(N)}$ are also provided.
\end{abstract}

\maketitle

\newcommand{\ba}{\begin{eqnarray}}
\newcommand{\ea}{\end{eqnarray}}
\newcommand{\ift}{\int_{0}^{\infty}}
\newcommand{\nn}{\nonumber}
\newcommand{\no}{\noindent}
\newcommand{\lf}{\left\lfloor}
\newcommand{\rf}{\right\rfloor}
\newcommand{\realpart}{\mathop{\rm Re}\nolimits}
\newcommand{\imagpart}{\mathop{\rm Im}\nolimits}

\newcommand{\op}[1]{\ensuremath{\operatorname{#1}}}
\newcommand{\pFq}[5]{\ensuremath{{}_{#1}F_{#2} \left( \genfrac{}{}{0pt}{}{#3}
{#4} \bigg| {#5} \right)}}

\newtheorem{Definition}{\bf Definition}[section]
\newtheorem{Thm}[Definition]{\bf Theorem}
\newtheorem{Example}[Definition]{\bf Example}
\newtheorem{Lem}[Definition]{\bf Lemma}
\newtheorem{Cor}[Definition]{\bf Corollary}
\newtheorem{Prop}[Definition]{\bf Proposition}
\numberwithin{equation}{section}

\section{Introduction}
\label{sec-intro}

The Euler numbers $E_{n}$, defined by the generating function 
\begin{equation}
\frac{1}{\cosh z} = \sum_{n=0}^{\infty} E_{n} \frac{z^{n}}{n!}
\label{rec-10}
\end{equation}
\noindent
and the Euler polynomials $E_{n}(x)$ that generalize them 
\begin{equation}
\sum_{n=0}^{\infty} E_{n}(x) \frac{z^{n}}{n!} = \frac{2 e^{xz}}{e^{z}+1}
\label{rec-11}
\end{equation}
(\cite[9.630,9.651]{gradshteyn-2007a}) are examples of basic special 
functions.  It follows directly from the definition that 
$E_{n} = 0 $ for  $n$ odd.  Morever, the relation 
$E_{n} = 2^{n} E_{n} \left( \tfrac{1}{2} \right)$ follows by setting $x = \tfrac{1}{2}$ in  \eqref{rec-11}, 
replacing $z$ by $2z$ and  comparing with \eqref{rec-10}. 

Moreover, the identity 
\begin{equation}
\frac{2e^{xz}}{e^{z}+1} = \frac{ 2 e^{ \left( x - 1/2  \right) z }} { e^{z/2} + e^{-z/2}} 
\end{equation}
\noindent
produces
\begin{equation}
E_{n}(x) = \sum_{k=0}^{n} \binom{n}{k} \frac{E_{k}}{2^{k}} \left( x - \tfrac{1}{2} \right)^{n-k} = 
\sum_{k=0}^{n} \binom{n}{k} E_{k} \left( \tfrac{1}{2} \right)  \left( x - \tfrac{1}{2} \right)^{n-k},
\label{euler-poly}
\end{equation}
\noindent
that gives  $E_{n}(x)$ in  terms of the Euler numbers (see \cite[9.650]{gradshteyn-2007a}). 

\smallskip

The \textit{generalized Euler polynomials} $E_{n}^{(p)}(z)$, defined by the generating function
\begin{equation}
\sum_{n=0}^{\infty} E_{n}^{(p)}(x) \frac{z^{n}}{n!} = 
\left( \frac{2}{1+ e^{z} } \right)^{p} e^{xz}, \quad \text{ for } p \in \mathbb{N}
\end{equation}
\noindent
are polynomials extending $E_{n}(x)$, the case $p=1$.  These appear in Section 
24.16 of \cite{olver-2010a}. The definition leads directly to 
the expression
\begin{equation}
\label{recu-1a}
E_{n}^{(p)}(x) = \sum_{k=0}^{n} \binom{n}{k} x^{k} E_{n-k}^{(p)}(0),
\end{equation}
\noindent 
where the \textit{generalized Euler numbers} $E_{n}^{(p)}(0)$ are defined recursively by
\begin{equation}
\label{recu-1b}
E_{n}^{(p)}(0) = \sum_{k=0}^{n} \binom{n}{k} E_{k}^{(p-1)}(0) E_{n-k}(0),
\end{equation}
for $p > 1$ and initial condition $E_{n}^{(1)}(0) = E_{n}(0)$.

\section{A probabilistic representation of Euler polynomials and their generalizations}
\label{sec-prob-def}

This section discusses probabilistic representations of the Euler polynomials and their 
generalizations.  The results involve the expectation operator $\mathbb{E}$ defined by 
\begin{equation}
\mathbb{E} g(L) = \int g(x) f_{L}(x) \, dx,
\end{equation}
\noindent
with $f_{L}(x)$ the probability density of the random variable $L$ and 
for any function $g$ such that the integral exists. 

\begin{proposition}
Let $L$ be a random variable with hyperbolic secant density 
\begin{equation}
f_{L}(x) =   \text{sech }\pi x, \quad \text{ for } x \in \mathbb{R}.
\end{equation}
\noindent
Then the Euler polynomial is given by
\begin{equation}
E_{n}(x) = \mathbb{E} \left( x + \imath L - \tfrac{1}{2} \right)^{n}.
\label{rep-1}
\end{equation}
\end{proposition}
\begin{proof}
The right hand-side of \eqref{rep-1} is 
\begin{eqnarray*}
\mathbb{E} \left( x + \imath L - \tfrac{1}{2} \right)^{n} & = &  \int_{-\infty}^{\infty} 
\left( x - \tfrac{1}{2} + \imath  t \right)^{n} \text{sech }\pi t \, dt \\
& = & \sum_{j=0}^{n}  \binom{n}{j} \left( x - \tfrac{1}{2} \right)^{n-j} \imath^{j} 
\int_{-\infty}^{\infty} t^{j} \text{sech }\pi  \, dt 
\end{eqnarray*}
\noindent
The identity 
\begin{equation}
\int_{-\infty}^{\infty} t^{k} \text{sech } \pi t \, dt =  \frac{|E_{k}|}{2^{k}} 
\end{equation}
\noindent
holds for $k$ odd, since both sides vanish and for $k$ even, it 
appears as entry $3.523.4$ in \cite{gradshteyn-2007a}. A proof of this entry may be 
found in \cite{boyadzhiev-2013a}. Then, using $|E_{2n}| = (-1)^{n} E_{2n}$ (entry $9.633$ in \cite{gradshteyn-2007a})
\begin{equation}
\mathbb{E} ( x + \imath L - \tfrac{1}{2} )^{n} = \sum_{j=0}^{n} \binom{n}{j} (x  - \tfrac{1}{2} )^{n-j} 
\frac{E_{j}}{2^{j}}  = E_{n}(x).
\end{equation}
\end{proof}

There is a natural extension to the case of $E_{n}^{(p)}(x)$.  The proof is similar to 
the previous case, so it is omitted. 

\begin{theorem}
\label{thm-general}
Let  $p \in \mathbb{N}$ and $L_{j}, \, 1 \leq j \leq p$ a collection of independent 
identically distributed random variables with hyperbolic secant distribution. Then 
\begin{equation}
E_{n}^{(p)}(x) = \mathbb{E} \left[ x + \sum_{j=1}^{p}  \left( \imath L_{j} - \tfrac{1}{2} \right) \right]^{n}.
\end{equation}
\end{theorem}

\smallskip

In a recent paper, L.~B.~Klebanov et al. \cite{klebanov-2012a} considered random 
sums of independent random variables of the form 
\begin{equation}
\frac{1}{N} \sum_{j=1}^{\mu_{N}} L_{j}
\end{equation}
\noindent
where the random number of summands $\mu_{N}$ is independent of the $L{j}'s$ and 
is described below.

\begin{definition}
\label{def-mu}
Let $N \in \mathbb{N}$ and $T_{N}(z)$ be the Chebyshev polynomial of the first kind. 
The random variable $\mu_{N}$ taking  values in $\mathbb{N}$, is defined by its  generating 
function
\begin{equation}
\label{def-mu1}
\mathbb{E} z^{\mu_{N}} = \frac{1}{T_{N}(1/z)}.
\end{equation}
\end{definition}

Information about the Chebyshev polynomials appears  in \cite{gradshteyn-2007a}
and \cite{olver-2010a}. 

\begin{example}
\label{prob-N2}
Take $N=2$. Then $T_{2}(z) = 2z^{2}-1$ gives
\begin{equation}
\mathbb{E} z^{\mu_{2}} = \frac{1}{T_{2}(1/z)} = \frac{z^{2}}{2-z^{2}} = 
\sum_{\ell=1}^{\infty} \frac{z^{2 \ell}}{2^{\ell}}.
\end{equation}
Therefore $\mu_{2}$ takes the value $2 \ell$, with $\ell \in \mathbb{N}$, with probability
\begin{equation}
\Pr(\mu_{2} = 2 \ell ) = 2^{-\ell}.
\end{equation}
\end{example}

In \cite{klebanov-2012a}, Klebanov et al. prove the following result. 

\begin{theorem}[Klebanov et al.]
Assume $\{ L_{j} \}$ is a sequence of independent identically distributed  random 
variables with hyperbolic secant distribution. Then, for all $N \geq 2$ and $\mu_{N}$ defined 
in \eqref{def-mu1}, the random variable 
\begin{equation}
L := \frac{1}{N} \sum_{j=1}^{\mu_{N}} L_{j}
\label{def-L}
\end{equation}
\noindent
has the same  hyperbolic secant distribution.
\end{theorem}

\section{The Euler polynomials in terms of the generalized ones}
\label{sec-expression}

The identifies \eqref{recu-1a} and \eqref{recu-1b} can be used to express the generalized Euler 
polynomial $E_{n}^{(p)}(x)$ in terms of the standard Euler polynomials $E_{n}(x)$. However, to the 
best of our knowledge, there is no formula that allows to express $E_{n}(x)$ in terms of 
$E_{n}^{(p)}(x)$. This section presents such a formula.  

\begin{definition}
Let $N \in \mathbb{N}$. The sequence $\{ p_{\ell}^{(N)}: \, \ell =0, \, 1, \, \cdots \}$ is defined as the coefficients 
in the expansion 
\begin{equation}
\frac{1}{T_{N}(1/z)} = \sum_{\ell=0}^{\infty} p_{\ell}^{(N)} z^{\ell}.
\end{equation}
\end{definition}

Definition \ref{def-mu} shows that 
\begin{equation}
p_{\ell}^{(N)} = \Pr(\mu_{N} = \ell), \quad \text{ for } \ell \in \mathbb{N}.
\end{equation}
\noindent
The numbers $p_{\ell}^{(N)}$ will be referred as the \textit{probability numbers}.  

\begin{example}
For $N=2$, Example \ref{prob-N2} gives
\begin{equation}
p_{\ell}^{(2)} = \begin{cases}
0 & \text{ if } \ell \text{ is odd } \\
2^{-\ell/2} & \text{ if } \ell \text{ is even}, \ell  \neq 0.
\end{cases}
\end{equation}
\end{example}

\smallskip

The coefficients $p_{\ell}^{(N)}$ are now used to
produce expansions of $E_{n}(x)$, one for each $N \in \mathbb{N}$,  in terms of the generalized 
Euler polynomials.

\begin{theorem}
\label{thm-eu1}
The Euler polynomials satisfy, for all $N \in \mathbb{N}$,
\begin{equation}
E_{n}(x)  =  \frac{1}{N^{n}} \mathbb{E} \left[ E_{n}^{(\mu_{N})} 
\left( \tfrac{1}{2} \mu_{N} + N (x - \tfrac{1}{2} ) \right) \right]
\end{equation}
\end{theorem}
\begin{proof}
From  \eqref{rep-1} and \eqref{def-L} 
\begin{equation}
E_{n} \left( \tfrac{1}{2} \right) = 
\mathbb{E} (\imath L)^{n} = \frac{1}{N^{n}} 
\mathbb{E} \left[ \imath \sum_{j=1}^{\mu_{N}} L_{j} \right]^{n},
\end{equation}
with Theorem \ref{thm-general}, this yields
\begin{equation}
\mathbb{E} \left[ E_{n}^{(\mu_{N})} \left( \frac{\mu_{N}}{2} 
\right) \right]  = \mathbb{E} \left[ \imath  \sum_{j=1}^{\mu_{N}}  L_{j}  \right]^{n} 
= N^{n} E_{n} \left( \tfrac{1}{2} \right).
\end{equation}
%\begin{equation}
%E_{n} \left( \tfrac{1}{2} \right) = 
%\frac{1}{N^{n}} \mathbb{E} \left[ E_{n}^{(\mu_{N})} \left( \frac{\mu_{N}}{2} 
%\right) \right].
%\end{equation}
\noindent
Using identity \eqref{euler-poly}, it follows that
\begin{eqnarray*}
E_{n}(x) & = &  \sum_{k=0}^{n} \binom{n}{k} E_{k} \left( \tfrac{1}{2} \right) 
\left( x - \tfrac{1}{2} \right)^{n-k} \\
& = & \mathbb{E} \left[ \sum_{k=0}^{n} \binom{n}{k} N^{-k} 
E_{k}^{(\mu_{N})} \left( \tfrac{1}{2}\mu_{N} \right) 
\left( x - \tfrac{1}{2} \right)^{n-k} \right] \\
& = & \mathbb{E} \left[ \sum_{k=0}^{n} \binom{n}{k} N^{-k} 
( \imath L_{1} + \cdots + \imath L_{\mu_{N}} )^{k} \left( x - \tfrac{1}{2} \right)^{n-k} 
\right] \\
& = & \mathbb{E} \left[ \frac{1}{N^{n}}\sum_{k=0}^{n} \binom{n}{k}  
( \imath  L_{1} + \cdots +  \imath L_{\mu_{N}} )^{k} \left( N(x - \tfrac{1}{2}) 
\right)^{n-k} \right]  \\
& = & \mathbb{E} \left[ \frac{1}{N^{n}} \left(  
\imath  L_{1} + \cdots + \imath L_{\mu_{N}}  +  N(x - \tfrac{1}{2}) 
\right)^{n} \right]  \\
& = & \mathbb{E} \left[ \frac{1}{N^{n}} \left(  
\imath  L_{1} + \cdots + \imath L_{\mu_{N}}  +  z - \tfrac{1}{2} \mu_{N} 
\right)^{n} \right]  \\
& = & \frac{1}{N^{n}} \mathbb{E} \left[ E_{n}^{(\mu_{N})}(z) \right],
\end{eqnarray*}
\noindent
where $z = \tfrac{1}{2} \mu_{N} + N \left( x - \tfrac{1}{2} \right)$. This 
completes the proof.
\end{proof}

The next result is established using the fact that the expectation operator $\mathbb{E}$ satisfies 
\begin{equation}
\mathbb{E}[ h(\mu_{N}) ] = \sum_{k=0}^{\infty} p_{k}^{(N)} h(k), 
\end{equation}
\noindent 
for any function $h$ such that the right-hand side exists.

\begin{corollary}
\label{coro-eu1}
The Euler polynomials satisfy 
\begin{equation}
E_{n}(x)  =   \frac{1}{N^{n}} \sum_{k=N}^{\infty} p_{k}^{(N)} 
E_{n}^{(k)} \left( \tfrac{1}{2}k + N \left( x  - \tfrac{1}{2} \right) \right).
\label{form-exp1}
\end{equation}
\end{corollary}

\begin{note}
 Corollary \ref{coro-eu1} gives an infinite family of expressions 
for $E_{n}(x)$ in terms of the generalized Euler polynomials $E_{n}^{(k)}(x)$, one for each value 
of $N \geq 2$.
\end{note}

\begin{example}
The expansion \eqref{form-exp1} with $N=2$ gives
\begin{equation}
\label{expan-N2}
E_{n}(x) = \frac{1}{2^{n}} \sum_{\ell=1}^{\infty} \frac{1}{2^{\ell}} E_{n}^{(2 \ell)}( \ell + 2x - 1).
\end{equation}
\noindent
For instance, when $n=1$, 
\begin{equation}
E_{1}(x) = \frac{1}{2} \sum_{\ell=1}^{\infty} \frac{1}{2^{\ell}} E_{1}^{(2 \ell)} ( \ell + 2x - 1 )
\end{equation}
\noindent
and the value $E_{1}^{(\ell)}(x) = x - \frac{\ell}{2}$ gives 
\begin{equation}
E_{1}(x) = \frac{1}{2} \sum_{\ell=1}^{\infty} \frac{1}{2^{\ell}} ( \ell + 2x - 1 - \ell) = x - \tfrac{1}{2}
\end{equation}
\noindent
as expected.
\end{example}

\section{The probability numbers}
\label{sec-probnum}

For fixed $N \in \mathbb{N}$, the random variable $\mu_{N}$ has been defined by its moment 
generating function 
\begin{equation}
\mathbb{E} z^{\mu_{N}} = \frac{1}{T_{N}(1/z)} = \sum_{\ell=0}^{\infty} p_{\ell}^{(N)} z^{\ell}.
\end{equation}
\noindent
This section presents properties of the probability numbers $p_{\ell}^{(N)}$ that appear in 
Corollary \ref{coro-eu1}. 

\smallskip

For small $N$, the coefficients $p_{\ell}^{(N)}$ can be computed directly by expanding the rational function 
$1/T_{N}(1/z)$ in partial fractions.  Example 
\ref{prob-N2} gave the case $N=2$. The cases   $N=3$ and $N=4$ are presented below.

\begin{example}
For $N=3$, the Chebyshev polynomial is 
\begin{equation}
T_{3}(z) = 4 z^{3} - 3z = 4z (z - \alpha) (z + \alpha),
\end{equation}
\noindent 
with $\alpha = \sqrt{3}/2$. This  yields
\begin{equation}
\frac{1}{T_{3}(1/z)} = \frac{z^{3}}{4(1 - \alpha z)(1 + \alpha z)} = 
\sum_{k=0}^{\infty} \frac{3^{k}}{2^{2k+2}} z^{2k+3}.
\end{equation}
\noindent
It follows that $p_{\ell}^{(3)} = 0 $ unless $\ell = 2k+3$ and 
\begin{equation}
p_{2k+3}^{(3)} = \frac{3^{k}}{2^{2k+2}}.
\end{equation}
Corollary \ref{coro-eu1} now gives 
\begin{equation}
E_{n}(x) = \frac{1}{3^{n}} \sum_{k=0}^{\infty} \frac{3^{k}}{2^{2k+2}} 
E_{n}^{(2k+3)}(3x+k),
\end{equation}
\noindent
a companion to \eqref{expan-N2}. 
\end{example}

\begin{example}
The probability numbers for $N=4$ are computed from the expression
\begin{equation}
\frac{1}{T_{4}(1/z)} = \frac{z^{4}}{z^{4}-8z^{2}+8}.
\end{equation}
\noindent
The factorization 
\begin{equation}
z^{4}-8z^{2}+8 = (z^{2}- \beta)(z^{2}-\gamma)
\end{equation}
\noindent
with $\beta = 2 (2 + \sqrt{2})$ and $\gamma  = 2 ( 2 - \sqrt{2})$ and the partial 
fraction decomposition
\begin{equation}
\frac{z^{4}}{z^{4}- 8z^{2} + 8} = \frac{\beta}{\beta - \gamma} \frac{1}{1 - \beta/z^{2}} -
\frac{\gamma}{\beta - \gamma} \frac{1}{1 - \gamma/z^{2}}
\end{equation}
\noindent
show that $p_{\ell}^{(4)} = 0$ for $\ell$ odd or $\ell = 2$ and 
\begin{equation}
p_{2 \ell}^{(4)} = \frac{\sqrt{2}}{2^{2 \ell + 1}} \left[ ( 2 + \sqrt{2})^{\ell -1} - 
(2 - \sqrt{2})^{\ell -1} \right]
\end{equation}
\noindent
for $\ell \geq 2$.  Corollary \ref{coro-eu1} now gives 
\begin{equation}
E_{n}(x) = \sqrt{2} \sum_{\ell=2}^{\infty} 
\frac{ \left[ ( 2 + \sqrt{2})^{\ell -1} - 
(2 - \sqrt{2})^{\ell -1} \right]}{2^{2 \ell + 1}} 
E_{n}^{(2 \ell)} (4x + \ell - 2 ).
\end{equation}
\end{example}

Some elementary properties  of the probability numbers are presented next. 

\begin{proposition}
The probability numbers $p_{\ell}^{(N)}$ vanish if $\ell < N$.
\end{proposition}
\begin{proof}
The Chebyshev polynomial $T_{N}(z)$ has the form $2^{N-1}z^{N} + $ lower 
order terms. Then the expansion of $1/T_{N}(1/z)$ has a zero of order $N$ at 
$z=0$. This proves the statement.
\end{proof}

\begin{proposition}
The probability numbers $p_{\ell}^{(N)}$ vanish if $\ell \not \equiv N \bmod 2$. 
\end{proposition}
\begin{proof}
The polynomial $T_{N}(z)$ has the same parity as $N$. The same holds for the 
rational function $1/T_{N}(1/z)$. 
\end{proof}

An expression for the probability numbers is given next.

\begin{theorem}
\label{expression-one}
Let $N \in \mathbb{N}$ be fixed and define 
\begin{equation}
\theta_{k}^{(N)} = \frac{ (2k-1) \pi}{2N}.
\end{equation}
\noindent
Then 
\begin{equation}
p_{\ell}^{(N)} = \frac{1}{N} \sum_{k=1}^{N} (-1)^{k+1} \sin \theta_{k}^{(N)} 
\cos^{\ell -1} \theta_{k}^{(N)}.
\end{equation}
\end{theorem}
 \begin{proof}
The Chebyshev polynomial is defined by $T_{N}(\cos \theta ) = 
\cos (N \theta)$, so its roots are $z_{k}^{(N)} = \cos \theta_{k}^{(N)}$, with
$\theta_{k}^{(N)}$ as above. The leading coefficient of $T_{N}(z)$ is $2^{N-1}$, thus 
\begin{equation}
\frac{1}{T_{N}(z)}  = \frac{2^{1-N}}{\prod_{k=1}^{N} ( z - z_{k})}.
\end{equation}
\noindent
In the remainder of the proof, the superscript $N$ has been dropped from $z_{k}^{(N)}$ and 
$\theta_{k}^{(N)}$,  for clarity.  Define 
\begin{equation}
Q(z) = \prod_{k=1}^{N} (z - z_{k}).
\end{equation}
\noindent
The roots $z_{k}$ of $Q$ are distinct, therefore 
\begin{equation}
\label{exp-Q}
\frac{1}{Q(z)} = \sum_{k=1}^{N} \frac{1}{Q'(z_{k})} \frac{1}{z-z_{k}}.
\end{equation}
\noindent
The identity $T_{N}'(z) = NU_{N-1}(z)$ gives 
\begin{equation}
Q'(z_{k}) = N2^{1-N} U_{N-1}(z_{k})
\end{equation}
\noindent 
where $U_{j}(z)$ is the Chebyshev polynomial of the second kind defined by
\begin{equation}
U_{N}( \cos \theta)  = \frac{\sin(N+1) \theta}{\sin \theta}. 
\end{equation}
\noindent 
Then 
\begin{equation}
U_{N-1}(z_{k}) = U_{N-1}(\cos \theta_{k}) = \frac{\sin N \theta_{k}}
{\sin \theta_{k}}. 
\end{equation}
\noindent
and the value $\sin N \theta_{k} = (-1)^{k+1}$ yields 
\begin{equation}
Q'(z_{k}) = \frac{(-1)^{k+1}}{\sin \theta_{k}} N 2^{1-N}.
\end{equation}
\noindent 
Therefore \eqref{exp-Q} now  gives 
\begin{equation}
\frac{1}{Q(z)} = \frac{2^{N-1}}{N} \sum_{k=1}^{N} \frac{(-1)^{k+1} 
\sin \theta_{k}}{z - \cos \theta_{k}}.
\end{equation}
\noindent
It follows that 
\begin{eqnarray*}
\frac{1}{T_{N}(1/z)}  =   \frac{2^{1-N}}{Q(1/z)} 
& = & \frac{1}{N} \sum_{k=1}^{N} (-1)^{k+1} \frac{z \sin \theta_{k}}{1 - z \cos 
\theta_{k}} \\
& = & \frac{1}{N} \sum_{k=1}^{N} (-1)^{k+1} \sin \theta_{k} 
\sum_{\ell = 0 }^{\infty} z^{\ell+1} \cos^{\ell} \theta_{k} \\
& = & \frac{1}{N} \sum_{\ell = 0}^{\infty} z^{\ell +1} 
\sum_{k=1}^{N} (-1)^{k+1} \sin \theta_{k} \cos^{\ell} \theta_{k}.
\end{eqnarray*}
\noindent
The proof is complete.
\end{proof}

The next result provides another explicit formula for the probability numbers.  The coefficients 
$A(n,k)$ appear in OEIS entry A008315, as entries of the Catalan triangle.

\smallskip

\begin{theorem}
\label{formula-sum}
Let $A(n,k) = \binom{n}{k} - \binom{n}{k-1}$. Then, if $N \equiv \ell 
\bmod 2$, 
\begin{equation*}
p_{\ell}^{(N)} = \frac{1}{2^{\ell}} 
\sum_{t = \lf \tfrac{1}{2} \left( \frac{2 - \ell}{N}-1 \right) \rf}
^{\lf \tfrac{1}{2} \left( \frac{\ell}{N}-1 \right)\rf}
(-1)^{t} A(\ell-1, \tfrac{1}{2} ( \ell - (2t+1)N)),
\end{equation*}
\no indent
when $\ell$ is not an odd multiple of $N$ and 
\begin{equation*}
p_{\ell}^{(N)} = \frac{1}{2^{\ell}} 
\left[ \sum_{s=1}^{\lf \ell/N-1 \rf} (-1)^{k - s} A(\ell-1,sN)
\right] + \frac{(-1)^{k}}{2^{\ell-1}}, 
\text{ with } k = \tfrac{1}{2} \left( \ell/N-1 \right)
\end{equation*}
\noindent
otherwise.
\end{theorem}

The proof begins with a preliminary result.

\begin{lemma}
Let $N \in \mathbb{N}$ and $\theta_{k} = \frac{\pi}{2} \frac{(2k-1)}{N}$. Then 
\begin{equation}
f_{N}(z) = \sum_{k=1}^{N} (-1)^{k+1} e^{\imath  \theta_{k} z}
\end{equation}
\noindent
is given by 
\begin{equation}
f_{N}(z) = \frac{1- (-1)^{N} e^{\pi \imath z}}{2 \cos \left( \frac{\pi z}{2N} 
\right) } \quad \text{ if } z \neq (2t+1)N \text{ with } t \in \mathbb{Z}
\label{case1}
\end{equation}
\noindent
and 
\begin{equation}
f_{N}(z) = (-1)^{t} N \imath \quad \text{ if } z = (2t+1)N \text{ for some } 
t \in \mathbb{Z}.
\label{case2}
\end{equation}
\noindent
In particular
\begin{equation}
f_{N}(k) = \begin{cases}
(-1)^{(k/N - 1)/2}N \imath  & \quad \text{ if }
\tfrac{k}{N} \text{ is an odd integer } \\
\frac{ 1 - (-1)^{N+k}}{2 \cos \left( \frac{\pi k}{2N} \right) }
& \quad \text{ otherwise}. 
\end{cases}
\end{equation}
\end{lemma}
\begin{proof}
The function $f_{N}$ is the sum of a geometric progression. The formula 
\eqref{case2} comes from \eqref{case1} by passing to the limit. 
\end{proof}

The proof of Theorem \ref{formula-sum} is given now. 

\smallskip

\begin{proof}
The expression for $p_{\ell}^{(N)}$ given in Theorem \ref{expression-one}
yields 
\begin{eqnarray*}
p_{\ell}^{(N)} & = &  \frac{1}{N} \sum_{k=1}^{N} (-1)^{k+1} 
\frac{( e^{\imath \theta_{k}} - e^{- \imath  \theta_{k}})}{2 i } 
\left( \frac{e^{\imath  \theta_{k}} + e^{- \imath  \theta_{k}} }{2} \right)^{\ell-1} \\
& = & 
\frac{1}{2^{\ell} N i } 
\sum_{k=1}^{N} (-1)^{k+1} \sum_{r=0}^{\ell -1} \binom{\ell - 1}{r} 
\left[ e^{\imath( \ell - 2 r  )\theta_{k}} - e^{\imath(\ell - 2r-2) \theta_{k}} \right] \\
& = & \frac{1}{2^{\ell} N \imath } 
\sum_{r=0}^{\ell -1} \binom{\ell-1}{r} 
\left[ f_{N}(l-2r) - f_{N}(l-2r-2) \right] \\
& = & \frac{1}{2^{\ell} N  \imath } \left[ 
\sum_{r=1}^{\ell -1} A(\ell-1,r) f_{N}(\ell-2r) + 
f_{N}(\ell) - f_{N}(-\ell) \right]. 
\end{eqnarray*}
\noindent
Now $f_{N}(\ell) = f_{N}(-\ell) = 0$ if $\ell/N$ is not an odd integer. On the 
other hand, if $\ell = (2t+1)N$, with $t \in \mathbb{Z}$, then 
\begin{equation}
f_{N}(\ell) = (-1)^{t}N \imath  \text{ and } f_{N}(- \ell) = -(-1)^{t}N \imath.
\end{equation}
Thus
\begin{equation*}
f_{N}(\ell) - f_{N}(-\ell) = \begin{cases}
2N \imath  (-1)^{(\ell/N-1)/2}& \quad 
\text{ if } \ell  \text{ is an odd multiple of }N \\
0 & \quad \text{ otherwise}.
\end{cases}
\end{equation*}

The simplification of the previous expression for $p_{\ell}^{(N)}$ is divided 
in two cases, according to whether $\ell$ is an odd multiple of $N$ or not. 

\smallskip

\noindent
\textbf{Case 1}. Assume $\ell$ is not an odd  multiple of $N$. Then 
\begin{equation}
p_{\ell}^{(N)} =  \frac{1}{2^{\ell} N \imath } 
\sum_{r=0}^{\ell -1} A(\ell-1,r) f_{N}(\ell-2r).
\end{equation}
\noindent
Morever, 
\begin{equation}
f_{N}(\ell - 2r) = \begin{cases}
(-1)^{t} N \imath & \quad \text{ if } \frac{l - 2r}{N} = 2t+1 \\
0 & \quad \text{ otherwise}. 
\end{cases}
\end{equation}
\noindent
Therefore
\begin{equation}
p_{\ell}^{(N)} = \frac{1}{2^{\ell}} 
\sum_{\begin{stackrel}{t = \tfrac{1}{2} \left( \frac{2 - \ell}{N}-1 \right)}
{\ell - 
2r = (2t+1)N}\end{stackrel}}^{\tfrac{1}{2} \left( \frac{\ell}{N}-1 \right)}
(-1)^{t} A(\ell-1,r).
\end{equation}
\noindent
Observe that $\ell - (2t+1)N$ is always an even integer, thus the index $r$ 
may be eliminated from the previous expression to obtain
\begin{equation}
p_{\ell}^{(N)}  =   \frac{1}{2^{\ell}} 
\sum_{t = \lf \tfrac{1}{2} \left( \frac{2 - \ell}{N}-1 \right) \rf}
^{\lf \tfrac{1}{2} \left( \frac{\ell}{N}-1 \right)\rf}
(-1)^{t} A(\ell-1, \tfrac{1}{2} ( \ell - (2t+1)N)).
\end{equation}

\smallskip

\noindent
\textbf{Case 2}. Assume $\ell$ is an odd  multiple of $N$, say $\ell = 
(2k+1)N$. Then 
\begin{eqnarray*}
p_{\ell}^{(N)} & = &   \frac{1}{2^{\ell} N i } \left[ 
\sum_{r=0}^{\ell -1} A(\ell-1,r) f_{N}(\ell-2r) + 2Ni (-1)^{k} \right] \\
 & = &   \frac{1}{2^{\ell} N i } \left[ 
\sum_{r=0}^{\ell -1} A(\ell-1,r) f_{N}(\ell-2r) \right] 
+ \frac{(-1)^{k}}{2^{\ell-1}}.
\end{eqnarray*}

The term $f_{N}(\ell - 2r)$ vanishes unless $\ell - 2r$ is an odd multiple 
of $N$. Given that $\ell = (2k+1)N$, the term is non-zero provided 
$2r$ is an even multiple of $N$; say $r = sN$ for $s \in \mathbb{N}$. 
The range of $s$ is $1 \leq s \leq \frac{\ell-1}{N} = 2k+1 - \frac{1}{N}$.
This implies $1 \leq s \leq 2k = \ell/N-1$, and it follows that
\begin{equation*}
p_{\ell}^{(N)} = \frac{1}{2^{\ell}} 
\left[ \sum_{s=1}^{\ell/N-1} (-1)^{k - s} A(\ell-1,sN)
\right] + \frac{(-1)^{k}}{2^{\ell-1}}, 
\text{ with } k = \tfrac{1}{2} \left( \ell/N-1 \right).
\end{equation*}

\smallskip 

\noindent
The proof is complete.
\end{proof}

\begin{note}
The expression in Theorem \ref{formula-sum} shows that 
 $p_{\ell}^{(N)}$ is a rational number with a 
denominator a power of $2$ of exponent at most $\ell$.  Arithmetic 
properties of these coefficients will be described in  a future 
publication \cite{moll-2014a}. Moreover, the probability numbers 
$p_{\ell}^{(N)}$ appear in the description of a  random walk on $N$ sites. 
Details will appear in \cite{moll-2014a}.
\end{note}

\section{An asymptotic expansion}
\label{sec-asym}

The final result deals with the asymptotic behavior of the probability numbers 
$p_{\ell}^{(N)}$. 

\begin{theorem}
Let $\varphi_{N}(z) = \mathbb{E} \left[ z^{\mu_{N}} \right]$. Then, for fixed $z$ in 
the unit disk $|z|<1$, 
\begin{equation}
\varphi_{N}(z) \sim \left( \frac{z}{1 + \sqrt{1-z^{2}}} \right)^{N}, \text{ as } N \to \infty.
\end{equation}
\end{theorem}
\begin{proof}
The generating function satisfies 
\begin{equation}
\varphi_{N}(z)  = 1/T_{N}(1/z) = \frac{z^{N}}{2^{N-1}} \prod_{k=1}^{N} 
\left( 1 - z \, \cos \theta_{k}^{(N)} \right)^{-1}
\end{equation}
\noindent
with $\theta_{k}^{(N)}= (2k-1)\pi/2N$ as before. Then 
\begin{equation}
\log \varphi_{N}(z)  = \log 2 + N \log \frac{z}{2} - \sum_{k=1}^{N} \log \left( 1 - z \cos \theta_{k}^{(N)} \right).
\end{equation}
\noindent
The last sum is approximated by a Riemann integral
\begin{equation*}
\frac{1}{N} \sum_{k=1}^{N} \log \left( 1 - z \cos \theta_{k}^{(N)} \right) \sim \frac{1}{\pi}  
\int_{0}^{\pi} \log(1 - z \cos \theta) \, d \theta = 
\log \left( \frac{1 + \sqrt{1-z^{2}}}{2} \right).
\end{equation*}
\noindent
The last evaluation is elementary. It appears as entry $4.224.9$ in \cite{gradshteyn-2007a}.  It follows 
that 
\begin{equation}
\log \varphi_{N}(z) \sim \log 2 + N \log \left( \frac{z}{2} \right) - N \log \left( \frac{ 1 + \sqrt{1 - z^{2}} }{2} \right) 
\end{equation}
\no indent
and this is equivalent to the result.
\end{proof}

The function
\begin{equation}
A(z) = \frac{2}{1 + \sqrt{1-4z}}  = \sum_{n=0}^{\infty} C_{n}z^{n}
\end{equation}
\noindent
is the generating function for the Catalan numbers 
\begin{equation}
C_{n} = \frac{1}{n+1} \binom{2n}{n}.
\end{equation}

The final result follows directly from the expansion of  Binet's formula for Chebyshev polynomial
\begin{equation}
T_{N}(z) = \frac{ ( z - \sqrt{z^{2}-1} )^{N}  + ( z + \sqrt{z^{2}-1} )^{N}}{2}.
\end{equation}
\noindent 
Some standard notation is recalled. 
Given two sequences $ \mathbf{a} = \{ a_{n} \}, \, \mathbf{b} = \{ b_{n} \}$, their convolution $\mathbf{c} = \mathbf{a} * 
\mathbf{b}$ is the sequence $\mathbf{c} = \{c_{n} \}$, with 
\begin{equation}
c_{n} = \sum_{j=0}^{n} a_{j}b_{n-j}.
\end{equation}
\noindent 
The \textit{convolution power} $\mathbf{c}^{(*N)}$ is the convolution of $\mathbf{c}$ with itself, $N$ times. 

\begin{theorem}
For $N \in \mathbb{N}$ fixed, the first $N$ nonzero terms of the sequence $q_{\ell}^{(N)} = 2^{\ell-1} p_{\ell}^{(N)}$ 
agree with the first $N$ terms of the $N$-th convolution power $C_{n}^{(*N)}$ of the Catalan sequence:
\begin{equation*}
q_{N}^{(N)} = C_{0}^{(*N)}, \, q_{N+2}^{(*N)}=C_{1}^{(*N)}, \, \cdots, q_{N+2k}^{(N)} = C_{k}^{(*N)}, \cdots, q_{3N-2}^{(N)} = C_{N-1}^{(*N)}.
\end{equation*}
\noindent
In terms of generating functions, this is equivalent to 
\begin{equation}
\left( \sum_{n=0}^{\infty} C_{n} z^{2n+1} \right)^{N} - \sum_{\ell=0}^{\infty} q_{\ell}^{(N)}z^{\ell} \sim 2^{N} z^{3N}.
\end{equation}
\end{theorem}

\medskip

\noindent
\textbf{Acknowledgments}. The second author acknowledges the partial 
support of NSF-DMS 1112656. The first author is a graduate student, funded 
in part by the same grant.

%\bibliography{../../../AllRef/official}
%\bibliographystyle{plain}
%\end{document}

\end{document}